\documentclass[12pt]{amsart}

\usepackage{amsmath}
\usepackage{amsthm}
\usepackage{xfrac}
\usepackage{txfonts}
\usepackage{indentfirst}
\usepackage{mathrsfs}
\usepackage{mathtools}
\usepackage{amssymb}

\newtheorem{theorem}{Theorem}[section]
\newtheorem{lemma}{Lemma}[section]

\newtheorem{definition}{Definition}[section]

\newtheorem*{mainthm}{Main Theorem}

\newtheorem*{thmB'}{Theorem B'}
\newtheorem*{thmB''}{Theorem B''}

\numberwithin{equation}{section}

\newcommand{\Z}{\mathbb{Z}}

\title{Box dimension of the graphs of the  generalized Weierstrass-type functions}
\author{Haojie Ren}
\date{\today}
\address{School of Mathematical Sciences, Fudan University, No 220 Handan Road, Shanghai, China 200433}
\email{20110180012@fudan.edu.cn}
\begin{document}
\setlength{\parindent}{1em}

\maketitle

\begin{abstract}
For a  Lipschitz $\mathbb{Z}-$periodic function $\phi:\mathbb{R}\to \mathbb{R}^2$ satisfied that $\mathbb{R}^2\setminus\{\phi(x):x\in\mathbb{R}\}$ is not connected, an integer $b\ge 2$ and $\lambda\in (c/{b^{\frac12}},1)$, we prove the following  for the  generalized Weierstrass-type function  $W(x)=\sum\limits_{n=0}^{\infty}{{\lambda}^n\phi(b^nx)}$: the box dimension of its graph is equal to $3+2\log_b\lambda$, where $c$ is a constant depending on $\phi$.
\end{abstract}

\section{Introduction}
This paper concerns the box dimension of the graphs of  the  generalized Weierstrass-type functions 
\begin{equation}\label{eqn:Wtype}
W^{\phi}_{\lambda,b}(x)=\sum\limits_{n=0}^{\infty}{{\lambda}^n\phi(b^nx)},\,\, x \in \mathbb{R}
\end{equation}
where $b > 1$, $1/b< \lambda < 1$ and $\phi(x):\mathbb{R} \to \mathbb{R}^d$ is a non-constant $\mathbb{Z}$-periodic Lipschitz function. The first famous example 
$$\sum\limits_{n=0}^{\infty}{{\lambda}^n\cos(b^nx)}$$
was introduced by Weierstrass as a example of a  continuous but nowhere differentiable  function, see \cite{hardy1916weierstrass}. The graphs of $W^{\phi}_{\lambda,b}$ and related functions were studied as fractal curves in fractal geometry starting from the work of Besicovicth and Ursell \cite{besicovitch1937sets}.

For  $d=1$ it is easy to check that $W^{\phi}_{\lambda,b}$ exhibits approximate self-affine with scales $\lambda$ and b, which suggests that its dimension should be equal  to 
$$D_1=2+\frac{\log\lambda}{\log b}.$$
Indeed, Kaplan et all.\cite{kaplan1984lyapunov} proved that in the case that $\phi$ is a trigonometric polynomial, either $W^{\phi}_{\lambda,b}$ is a $C^1$ curve or the box dimension of the graph of $W^{\phi}_{\lambda,b}$ is equal to $D_1$ (see also \cite{przytycki1989hausdorff} \cite{hu1993fractal}). In the literature there have been many works on or related to the Hausdorff dimension of the graph of the Weierstrass type functions (see e.g. \cite{mauldin1986hausdorff,przytycki1989hausdorff,ledrappier1992dimension,solomyak1995random,hochman2014self,varju2019dimension,barany2019hausdorff,mandelbrot1977fractals,hunt1998hausdorff,tsujii2001fat,shen2018hausdorff}).
 Recently, for the case $\phi$ is real analytic, thanks to the new theories in \cite{hochman2014self,barany2019hausdorff}, Shen and the author \cite{ren2021dichotomy} proved the following dichotomy for $W^{\phi}_{\lambda,b}$: Either $W^{\phi}_{\lambda,b}$ is real analytic, or the Hausdorff dimension of its graph is equal to $D_1$.

In this paper we consider the case $d=2$, i.e. $\phi:\mathbb{R}\to \mathbb{R}^2$ is a Lipschitz $\mathbb{Z}-$periodic function. By the similar observation as $d=1$, it is natural to conjecture that, unless $W^{\phi}_{\lambda,b}$ is Lipschitz, the box and Hausdorff dimension of its graph should be  eaqual to 
$$D_2=\begin{cases}
\frac{\log b}{\log\lambda^{-1}},&b\lambda^2<1;\\
3+2\frac{\log\lambda}{\log b},&b\lambda^2\ge1.
\end{cases}
$$
Bara\'nski \cite{Baranski} proved that for $\phi(x)=e^{2\pi ix}$ the box dimension of the graph of the function $W^{\phi}_{\lambda,b}$ 
is equal to $3+2\frac{\log\lambda}{\log b}$, provided $b\in\mathbb{N}$, $b\ge2$ and $\lambda<1$ is sufficiently close to 1. Inspired by the approach in  \cite{Baranski}, we generalize this result to more  general functions $\phi$ by some planar topological methods. More precisely, we proved that

\begin{mainthm}Let $\phi:\mathbb{R}\to\mathbb{R}^2$ be a $\Z$-periodic  Lipschitz function. If the set  $\mathbb{R}^2\setminus\{\phi(s):s\in\mathbb{R}\}$ is not connected, then there is a number $c>0$ depending only on $\phi$ such that the following holds. For any integer $b>1$ and $0<\lambda<1$ with $b\lambda^2>c$, the graph of $W^{\phi}_{\lambda,b}$ has box dimension equal to  $3+2\log_b\lambda$.
\end{mainthm}

 However we can't expect that the above $c$ equal to 1 for every such $\phi$. Since for any integer $b>1,0<\lambda<1$ and $b\lambda^2>1$, we can consider $W_0(x)=e^{2\pi ix}$
and $\phi(x)=W_0(x)-\lambda W_0(bx)$, one has $W^{\phi}_{\lambda,b}(x)=W_0(x)$ and   the set  $\mathbb{R}^2\setminus\{\phi(s):s\in\mathbb{R}\}$ is not connected. For the case $\phi(x)=e^{2\pi ix}$, we believe the box dimension of the graph of complex Weierstrass funciton $W_{\lambda,b}^{e^{2\pi is}}$ is $3+2\frac{\log\lambda}{\log b}$ for $b\lambda^2>1$, unfortunately our method can only  prove it for $b\lambda^2>8\pi$. For simplicity we just write $W^{\phi}_{\lambda,b}$ as W in the rest of this paper.

\section{Preliminaries}\label{sec:Pre}
This section is devoted to give the basic formula for proving our Main Theorem. The observations are from  \cite{Baranski}, but we rewrite here for the reader's convenience. Let $L=L(\phi)$ be the Lipschitz constant of $\phi$, which means
$$L=\sup_{a\neq b\in\mathbb{R}}\frac{|\phi(a)-\phi(b)|}{|a-b|}.$$
To give the specific value of the constant c decided by $\phi$ in Main Theorem, we need the following notation, for any $t>0$ and $z\in\mathbb{R}^2$, let 

$$z=\boldsymbol{O}(t)$$
mean that $|z|\le t$.
For any $n\in\mathbb{N},\,k=0,1,\ldots,b^n-1,$ let
$$z_{n, k}=\frac{k}{b^n}.$$
Let $\varLambda=\{0,1,\ldots,b-1\}$.
\begin{lemma}\label{mainformula}
	For integer $n\ge1,k=0,1,\ldots,b^n-1$ and $j\in\varLambda$, we have
	\begin{equation}\label{formula1}
	\frac{W(z_{n+1, kb+j})-W(z_{n, k})}{\lambda^n}=\phi\left(\frac jb\right)-\phi(0)+\boldsymbol{O}\left(\frac{L\gamma}{1-\gamma}\right)
	\end{equation}
		\begin{equation}\label{formula2}
	\frac{W(z_{n+1, kb+j})-W(z_{n, k})}{\lambda^n}=\phi\left(\frac jb\right)-\phi(0)+\lambda^{-1}\bigg(\phi\left(\frac kb+\frac j{b^2}\right)-\phi\left(\frac k b\right)\bigg)+\boldsymbol{O}\left(\frac{L\gamma^2}{1-\gamma}\right)
	\end{equation}
	where $\gamma=\frac1{b\lambda}.$
\end{lemma}
\begin{proof}
	We only prove  (\ref{formula1}), since the proof of (\ref{formula2}) is similar. By the definition of W and the $\mathbb{Z}$-periodicity of $\phi$, we have 
	\begin{equation}
	\begin{split}
	W(z_{n+1, kb+j})-W(z_{n, k})     &=\sum_{t=0}^{n}\lambda^t\bigg(\phi\left(\frac{kb+j}{b^{n-t+1}}\right)-\phi\left(\frac{k}{b^{n-t}}\right)\bigg) \\
	     &=\lambda^n\sum_{t=0}^{n}\lambda^{-t}\bigg(\phi\left(\frac{kb+j}{b^{t+1}}\right)-\phi\left(\frac{k}{b^{t}}\right)\bigg)\\
	     &=\lambda^n\bigg(\phi\left(\frac jb\right)-\phi(0)+\boldsymbol{O}\left(\frac{L\gamma}{1-\gamma}\right)\bigg).
	 \end{split}
	\end{equation}
	In the last equation we used that $\phi$ is a Lipschitz function for $t\ge1$, which implies
	$$\bigg| \phi\left(\frac{kb+j}{b^{t+1}}\right)-\phi\left(\frac{k}{b^{t}}\right)\bigg|\le\frac L{b^t}.$$

\end{proof}

\section{the case when $\lambda^3b>c_0$}\label{sec:entpor}
In this section, we will find  a constant $c_0>0$ depending on $\phi$ and prove a result  similar to  the Main Theorem for  $\lambda^3b>c_0$, since our method can't deal with the case $\lambda^2b>c$ directly when $\lambda$ is sufficiently close to 1. For set $A,B\subset\mathbb{R}^2$ and $\mu\in\mathbb{R}^2$, let 
$$A+B=\{a+b:a\in A,\, b\in B\}$$
$$\mu A=\{\mu a:a\in A\}.$$
We write $\boldsymbol{B}(z,r)$ for the open disc centred at $z\in\mathbb{R}^2$ of radius r. For a function $f:E\subset\mathbb{R}\to \mathbb{R}^2$, let $\hat{f}$ be the image set of $f$ and let $\Delta(f)$ be the oscillation of $f$ defined as 
$$\Delta(f)=\sup\limits_{a,b\in E}|f(a)-f(b)|.$$
 and 
 $$\varepsilon=\varepsilon(f)=\sup\{r>0:\boldsymbol{B}(0,r)\subset\big(\mathbb{R}^2\setminus\hat{f}\big)\setminus G_{\infty}\}$$
 where $G_{\infty}$ is the unbounded  component of $\mathbb{R}^2\setminus\hat{f}$ (If the above set is empty, let $\varepsilon=0$.)


\begin{lemma}\label{lem3.0}
	Let $\phi$ be a $\Z$-periodic complex Lipschitz function satisfied that $\hat{\phi}$ separates $B(0,\varepsilon)$ from $\infty$ , then 
	\begin{equation}\label{lem3.1}
	\hat{\phi}+\boldsymbol{B}(0,\varepsilon\lambda)\subset\bigcup_{s\in[0,1)}\bigg(\phi(s)+\lambda\hat{\phi}\bigg)\quad\forall\lambda\in(0,1).
	\end{equation}
	
\end{lemma}
\begin{proof}
 If (\ref{lem3.1}) fails, then there exists $\lambda_0\in(0,1)$, $p_0\in\boldsymbol{B}(0,\varepsilon\lambda_0)$ and $s_0\in[0,1)$ such that 
	$$\phi(s_0)+p_0\notin\bigcup_{s\in[0,1)}\bigg(\phi(s)+\lambda_0\hat{\phi}\bigg).$$
	Thus
	\begin{equation}\label{lem3.1.1}
	\frac{\phi(s_0)+p_0-\phi(s)}{\lambda_0}\notin\hat{\phi}\quad\forall s\in[0,1).
	\end{equation}
	Since $	\frac{\phi(s_0)+p_0-\phi}{\lambda_0}$ is a continuous function, the set of the  image  is contained in a connected component of $\mathbb{R}^2\setminus\hat{\phi}$. Since we also observe that 
	$$\frac{\phi(s_0)+p_0-\phi(s_0)}{\lambda_0}\in\boldsymbol{B}(0,\varepsilon),$$
	we have
	\begin{equation}\label{lem3.1.2}
	\frac{\phi(s_0)+p_0-\phi(s)}{\lambda_0}\in G\quad\forall s\in[0,1)
	\end{equation}
	where G is the connected component of $\mathbb{R}^2\setminus\hat{\phi}$ that contains 0
	.

	Since G is bounded open set and $\partial G\subset \hat{\phi}$, (\ref{lem3.1.2}) yield that
	\begin{equation}\label{lem3.1.3}
		\Delta\left(\frac{\phi(s_0)+p_0-\phi}{\lambda_0}\right)\le \Delta(\phi).
	\end{equation}

	By the definiton of $\Delta$, we also have 
	$$\Delta\left(\frac{\phi(s_0)+p_0-\phi}{\lambda_0}\right)=\frac1{\lambda_0}\Delta(\phi).$$
	We combine this with (\ref{lem3.1.3}) to get the contradicition.
	
\end{proof}

In the next lemma we apply lemma \ref{lem3.0} repeatedly to find a specific open set contained in the  image of the set $\{W(s):s\in\left[\frac k {b^n},\frac{k+1}{b^n}\right]\}$ for  integer $n\ge 1$ and $k=0,1,\ldots,b^n-1$. The strategy of the proof is similar to Bara{\'n}ski \cite{Baranski}. Let 
$c_0>1$ be the solution of the following equation 
$$c_0=\frac L{\varepsilon}\left(2+\frac1{c_0-1}\right).$$

\begin{lemma}\label{OpenSet}
	Assume  $\hat{\phi}$ separates $B(0,\varepsilon)$ from $\infty$ .
Let $b\ge2$ be an integer, $\lambda\in(\frac1b,1)$ and $\lambda^3b> c_0$. Then, for integer $n\ge1$ and $k=0,1,\ldots,b^n-1$, the following holds:
\begin{equation}\label{lem3.2.0}
W(z_{n,k})-\frac{\lambda^n}{1-\lambda}\phi(0)+\lambda^n\big(\hat{\phi}+\boldsymbol{B}(0,\varepsilon\lambda)\big)\subset\bigg\{W(s):s\in\left[\frac k {b^n},\frac{k+1}{b^n}\right]\bigg\}.
\end{equation}
\end{lemma}
\begin{proof}
	 For $s\in[0,1)$, there exists $j_s\in\varLambda$ such that $|s-\frac{j_s}b|<\frac1b$. Thus
	 $$\phi(s)\in\phi\left(\frac{j_s}b\right)+\boldsymbol{B}\left(0,\frac Lb\right).$$
	 Combining (\ref{lem3.1}) we obtain
	 	\begin{equation}\label{lem3.2.1}
	 \hat{\phi}+\boldsymbol{B}(0,\varepsilon\lambda)\subset\bigcup_{j\in\varLambda}\bigg(\phi\left(\frac jb\right)+\lambda\hat{\phi}+\boldsymbol{B}\left(0,\frac Lb\right)\bigg).
	 \end{equation}
	 Combining with  (\ref{formula1}), this implies
	 $$\hat{\phi}+\boldsymbol{B}(0,\varepsilon\lambda)\subset\bigcup_{j\in\varLambda}\bigg(	\frac{W(z_{n+1, kb+j})-W(z_{n, k})}{\lambda^n}+\phi(0)+\lambda\hat{\phi}+\boldsymbol{B}\left(0,\frac Lb+\frac{L\gamma}{1-\gamma}\right)\bigg).$$
	 Hence,
	 \begin{equation}\label{lem3.2.2}
	W(z_{n, k})+ \lambda^n\hat{\phi}+\boldsymbol{B}(0,\varepsilon\lambda^{n+1})-\lambda^n\phi(0)\subset\bigcup_{j\in\varLambda}\bigg(W(z_{n+1, kb+j})+\lambda^{n+1}\hat{\phi}+\boldsymbol{B}\left(0,\left(\frac 1b+\frac{\gamma}{1-\gamma}\right)L\lambda^n\right)\bigg).
	\end{equation}
	We also observe that
	$$b\lambda^3>c_0=\frac L{\varepsilon}\left(2+\frac1{c_0-1}\right)>\frac L{\varepsilon}\left(1+\frac1{1-\gamma}\right)>\frac L{\varepsilon}\left(\lambda+\frac1{1-\gamma}\right),$$
	since $\gamma<\frac1{b\lambda^3}<\frac1{c_0}$ and $\lambda\in(0,1)$,
	which implies 
	\begin{equation}\label{lem3.2.3}
	\varepsilon\lambda^2>L\left(\frac1b+\frac\gamma{1-\gamma}\right).
	\end{equation}
	We combine (\ref{lem3.2.2}) and (\ref{lem3.2.3}):
	\begin{equation}
	\label{lem3.2.5}
		W(z_{n, k})+ \lambda^n\hat{\phi}+\boldsymbol{B}(0,\varepsilon\lambda^{n+1})-\lambda^n\phi(0)\subset\bigcup_{j\in\varLambda}\bigg(W(z_{n+1, kb+j})+\lambda^{n+1}\hat{\phi}+\boldsymbol{B}(0,\varepsilon\lambda^{n+2})\bigg).
	\end{equation}
	Since (\ref{lem3.2.5}) holds for any integer $n\ge1$ and $k=0,1,\ldots,b^n-1$, we also have
	$$W(z_{n+1, kb+j})+ \lambda^{n+1}\hat{\phi}+\boldsymbol{B}(0,\varepsilon\lambda^{n+2})-\lambda^{n+1}\phi(0)\subset\bigcup_{t\in\varLambda}\bigg(W(z_{n+2, kb^2+jb+t})+\lambda^{n+2}\hat{\phi}+\boldsymbol{B}(0,\varepsilon\lambda^{n+3})\bigg),$$
	for any $j\in\varLambda,$ wich implies
	$$	W(z_{n, k})+ \lambda^n\hat{\phi}+\boldsymbol{B}(0,\varepsilon\lambda^{n+1})-(\lambda^n+\lambda^{n+1})\phi(0)\subset\bigcup_{j\in\{0,1,\ldots,b^2-1\}}\bigg(W(z_{n+2, kb^2+j})+\lambda^{n+2}\hat{\phi}+\boldsymbol{B}(0,\varepsilon\lambda^{n+4})\bigg).$$
	For any $m\in\mathbb{N},$ 
	repeating the this process for m times,
	$$	W(z_{n, k})+ \lambda^n\hat{\phi}+\boldsymbol{B}(0,\varepsilon\lambda^{n+1})-\lambda^n\frac{1-\lambda^m}{1-\lambda}\phi(0)\subset\bigcup_{j\in\{0,1,\ldots,b^m-1\}}\bigg(W(z_{n+m, kb^m+j})+\lambda^{n+m}\hat{\phi}+\boldsymbol{B}(0,\varepsilon\lambda^{n+m+2})\bigg).$$
    Letting  $m\to\infty$, we obtain (\ref{lem3.2.0}).
	
\end{proof}	
\begin{definition}
	Let  $b\ge2,\,d\ge1$  and n be an integer,
	The n generation partition of $\mathbb{R}^d$ into b-adic intervals is
 $$\mathcal{L}_n^{b,d}=\bigg\{\left[\frac{k_1}{b^n},\frac{k_1+1}{b^n}\right)\times\left[\frac{k_2}{b^n},\frac{k_2+1}{b^n}\right)\times\ldots\times\left[\frac{k_d}{b^n},\frac{k_d+1}{b^n}\right):k_1,k_2,\ldots,k_d\in\mathbb{Z}\bigg\}$$ 
 The n generation covering number of $A\subset\mathbb{R}^d$ by b-adic cubes is 
 $$N(A,\mathcal{L}_n^{b,d})=\#\big\{d\in\mathcal{L}_n^{b,d}:D\cap A \ne\emptyset\big\}.$$
 Recall that the lower and upper box dimension are defined as 
 $$\underline{dim}_b(A)=\liminf_{n\to\infty}\frac{\log N(A,\mathcal{L}_n^{b,d})}{\log b},\quad\overline{dim}_b(A)=\limsup_{n\to\infty}\frac{\log N(A,\mathcal{L}_n^{b,d})}{\log b}.$$
 \end{definition}
By using Lemma \ref{OpenSet}, the following proof is  standard, but we give the proof for the reader's convenience.

\begin{theorem}\label{Theorem2}
	Let $\phi:\mathbb{R}\to\mathbb{R}^2$ be a $\Z$-periodic Lipschitz function such that the set  $\mathbb{R}^2\setminus\hat{\phi}$ is not connected, then there is a number $c_0\large \ge1$ depending only on $\phi$ such that, for integer $b>1,0<\lambda<1$ and $b\lambda^3>c_0$, the graph of $W^{\phi}_{\lambda,b}$ has box dimension   $3+2\log_b\lambda$.
\end{theorem}
\begin{proof}
	Without loss of generality, we may assume that $\hat{\phi}$ separates $B(0,\varepsilon)$ from $\infty$. (Since there exists a dot $z_0\in\mathbb{R}^2$ such that $\boldsymbol{B}(z_0,\varepsilon)\subset\mathbb{R}^2\setminus\hat{\phi}$. We could consider $\phi-z_0$  instead of $\phi$.) Since W is H{\"o}lder continous with exponent $\log_b\frac1{\lambda}$, we obtain by \cite[Lemma2.2]{Baranski}
	$$\overline{dim}_{b}\varGamma W\le3+2\log_b\lambda$$
	where $\varGamma W=\{(s,W(s)):s\in[0,1)\}$ is the graph of W.
	
	For the other direction, by the definiton of covering number, we can write
	\begin{equation}\label{3.3.1}
		N(\varGamma W,\mathcal{L}_n^{b,3})=\sum_{k=0}^{b^n-1}N\bigg(\bigg\{W(s):s\in[\frac k{b^n},\frac{k-1}{b^n})\bigg\},\mathcal{L}_n^{b,2}\bigg)\quad\forall n\in\mathbb{N}.
	\end{equation}
	For sufficiently large n, Lemma \ref{OpenSet} implies
	$$N\bigg(\bigg\{W(s):s\in[\frac k{b^n},\frac{k-1}{b^n})\bigg\},\mathcal{L}_n^{b,2}\bigg)\ge\frac{\pi\varepsilon^2\lambda^{2+2n}b^{2n}}2.$$
	We combine this with (\ref{3.3.1}) to get
		$$\underline{dim}_{B}\varGamma W\ge3+2\log_b\lambda.$$
\end{proof}

\section{the case when $b\lambda^2>c$}
In this section, we shall complete the proof of the Main Theorem under the additional assumption that $\lambda\in(0,1/2)$. Otherwise let $c\ge2c_0$ and the result holds by Theorem \ref{Theorem2}, since $b\lambda^2\ge c$ implies $b\lambda^3\ge c_0$. The idea of the proof is similar to  what we did in section 3, but the argument is more complicated. For $\beta\in\mathbb{R}$, let
$$\ell_{\beta}(s)=\phi(s)+\lambda^{-1}\bigg(\phi(\beta+\frac sb)-\phi(\beta)\bigg)\quad\forall s\in[0,1).$$
\begin{lemma}\label{lem4.1}
	Let $\phi:\mathbb{R}\to\mathbb{R}^2$ be a $\Z$-periodic  Lipschitz function satisfied that $\hat{\phi}$ separates $B(0,\varepsilon)$ from $\infty$. For  $\beta\in[0,1)$ and $b\lambda>\frac L{D(\phi)(1-\lambda)}$, the following holds
	\begin{equation}\label{lem4.1.1}
	\hat{\ell_{\beta}}+\boldsymbol{B}(0,\varepsilon\lambda)\subset\bigcup_{s\in[0,1)}\bigg(\ell_{\beta}(s)+\lambda\hat{\phi}\bigg)\quad\forall\lambda\in(0,1).
	\end{equation}
	
\end{lemma}
\begin{proof}
	Arguing by contradiction, assume  (\ref{lem4.1.1}) fails. Then there exists $\lambda_0\in(0,1)$, $p_0\in\boldsymbol{B}(0,\varepsilon\lambda_0)$ and $s_0\in[0,1)$ such that 
	\begin{equation}\label{lem4.1.2}
		\ell_{\beta}(s_0)+p_0\notin\bigcup_{s\in[0,1)}\bigg(\ell_{\beta}(s)+\lambda_0\hat{\phi}\bigg).
	\end{equation}
	Similar to the proof in Lemma \ref{lem3.0}, (\ref{lem4.1.2}) implies
	\begin{equation}\label{lem4.1.3}
	\Delta\big(\frac{\ell_{\beta}(s_0)+p_0-\ell_{\beta}}{\lambda_0}\big)\le \Delta(\phi).
	\end{equation}
	We also observe that, for $s_1,\,s_2\in[0,1)$,
	$$\big|\ell_{\beta}(s_1)-\ell_{\beta}(s_2)\big|\ge\big|\phi(s_1)-\phi(s_2)\big|-\lambda^{-1}\big|\phi(\beta+\frac{s_1}b)-\phi(\beta+\frac{s_2}b)\big|\ge\big|\phi(s_1)-\phi(s_2)\big|-L\gamma.$$
	Therefore 
	\begin{equation}\label{lem4.1.4}
		\Delta\big(\ell_{\beta}(s_0)+p_0-\ell_{\beta}\big)= \Delta(\ell_{\beta})\ge \Delta(\phi)-L\gamma.
	\end{equation}
	We combine (\ref{lem4.1.3}) and (\ref{lem4.1.4}) and obtain
	$$\lambda_0 \Delta(\phi)\ge \Delta(\phi)-L\gamma,$$
	which contradicts our conditions.	
\end{proof}
Let 
$c_1>1$ be the solution of the following equation 
$$c_1=\frac L{\varepsilon}(4+\frac1{c_1-1}).$$
\begin{lemma}\label{OpenSet2}
	Assume $\hat{\phi}$ separates $B(0,\varepsilon)$ from $\infty$,
	Let $b\ge2$ be an integer, $\lambda\in(0,1/2)$ and $\lambda^2b> c_2=\max\{c_1,\frac{2L}{D(\phi)}\}$. Then, for integer $n\ge1$ and $k=0,1,\ldots,b^n-1$, the following holds:
	\begin{equation}\label{lem4.2.0}
	W(z_{n,k})-\frac{\lambda^n}{1-\lambda}\phi(0)+\lambda^n\big(\hat{\ell_{\frac kb}}+\boldsymbol{B}(0,\varepsilon\lambda)\big)\subset\bigg\{W(s):s\in\left[\frac k {b^n},\frac{k+1}{b^n}\right]\bigg\}.
	\end{equation}
	
\end{lemma}
\begin{proof}
	For $\beta,s\in[0,1)$, there exists $j_s\in\varLambda$ such that $|s-\frac{j_s}b|<\frac1b$. Note
	\begin{equation}\label{4.2.1}
	\bigg|\ell_{\beta}\left(\frac{j_s}b\right)-\ell_{\beta}(s)\bigg|\le\frac Lb+\frac{L\gamma}b\le\frac{2L}b.
	\end{equation}
	 We also have  
	\begin{equation}\label{4.2.2}
	\hat{\phi}\subset\hat{\ell}_{s}+\boldsymbol{B}(0,L\gamma)\quad\forall s\in\mathbb{R}.
	\end{equation}

	Combine (\ref{4.2.1}), (\ref{lem4.1.1}) and (\ref{4.2.2})  with 
	$$b\lambda>b\lambda^2>c_1\ge\frac{2L}{\Delta(\phi)}>\frac{L}{(1-\lambda)\Delta(\phi)}$$
	to obtain
	\begin{equation}\label{4.2.3}
	\hat{\ell}_{\beta}+\boldsymbol{B}(0,\varepsilon\lambda)\subset\bigcup_{j\in\varLambda}\bigg(\ell_{\beta}\left(\frac jb\right)+\lambda\hat{\ell}_{\frac jb}+\boldsymbol{B}\left(0,\frac{3L}b\right)\bigg).
	\end{equation} Therefore
	\begin{equation}\label{4.2.4}
	\hat{\ell}_{\beta}+\boldsymbol{B}(0,\varepsilon\lambda)\subset\bigcup_{j\in\varLambda}\bigg(\ell_{\beta}\left(\frac jb\right)+\lambda\hat{\ell}_{\frac jb}+\boldsymbol{B}(0,\varepsilon\lambda^2)\bigg)
	\end{equation}
	by using $b\lambda^2>c_1=\frac L{\varepsilon}\left(4+\frac1{c_1-1}\right)>3\frac L{\varepsilon}$.
    Combine (\ref{4.2.4}) with $\beta=k/b$ and (\ref{formula2}), we obtain
	\begin{equation}\label{4.2.5}
	\hat{\ell}_{\frac kb}+\boldsymbol{B}(0,\varepsilon\lambda)\subset\bigcup_{j\in\varLambda}\bigg(\frac{W(Z_{n+1,kb+j})-W(Z_{n,k})}{\lambda^n}+\phi(0)+\lambda\hat{\ell}_{\frac jb}+\boldsymbol{B}\left(0,\frac{3L}b+\frac{L\gamma^2}{1-\gamma}\right)\bigg)
	\end{equation}
	We also observe that
	\begin{equation}
	\label{4.2.6}
	b\lambda^2>c_1=\frac L{\varepsilon}\left(4+\frac1{c_1-1}\right)>\frac L{\varepsilon}\left(3+\frac1{1-\gamma}\right)>\frac L{\varepsilon}\left(3+\frac1{c_1(1-\gamma)}\right)>\frac L{\varepsilon}\left(3+\frac{b\gamma^2}{1-\gamma}\right),
	\end{equation}
	since $\gamma<\frac1{b\lambda^2}<\frac1{c_1}$ and $\frac1b>\gamma^2$.
	We combine (\ref{4.2.5}) and (\ref{4.2.6}):
	\begin{equation}
	\label{lem3.2.4}
	W(z_{n, k})+ \lambda^n\hat{\ell}_{\frac kb}+\boldsymbol{B}(0,\varepsilon\lambda^{n+1})-\lambda^n\phi(0)\subset\bigcup_{j\in\varLambda}\bigg(W(z_{n+1, kb+j})+\lambda^{n+1}\hat{\ell}_{\frac{kb+j}b}+\boldsymbol{B}(0,\varepsilon\lambda^{n+2})\bigg).
	\end{equation}
	The rest of the proof is the same as  we did in Lemma \ref{OpenSet}, so (\ref{lem4.2.0}) holds.
\end{proof}	
\begin{proof}[Proof] Without loss generation we may assume  $\hat{\phi}$ separates $B(0,\varepsilon)$ from $\infty$. Let $c=\max\{2c_0,c_1\}$. For the case $\lambda\in(1/2,1)$, we can use $c\ge2c_0$ and finish the proof by Theorem \ref{Theorem2} as explained at the beginning of this section. For the other case, the proof is the same as Theorem \ref{Theorem2} by Lemma \ref{OpenSet2}.
	
\end{proof}

\bibliographystyle{plain}             

\begin{thebibliography}{1}
\bibitem{Baranski}
K. Bara{\'n}ski.:
\newblock On the complexification of the Weierstrass non-differentiable function.
\newblock {\em Ann. Acad. Sci. Fenn. Math.} 27(2), 325-340 (2002)

\bibitem{hardy1916weierstrass}
Hardy, G.H.:
\newblock Weierstrass’s non-differentiable function.
\newblock {Trans. Amer. Math. Soc} 17(3), 301-325 (1916)

\bibitem{besicovitch1937sets}
Besicovitch, A.S., Ursell, H.D.:
\newblock Sets of fractional dimensions (V): On dimensional numbers of some continuous curves.
\newblock {\em J. Lond. Math. Soc.} 1(1), 18-25 (1937)

\bibitem{kaplan1984lyapunov}
Kaplan, J.L., Mallet-Paret,J., Yorke, J.A.:
\newblock The Lyapunov dimension of a nowhere differentiable attracting torus.
\newblock {\em Ergod. Theory Dyn. Syst.} 4(2), 261-281 (1984)

\bibitem{przytycki1989hausdorff}
Przytycki, F., Urba{\'n}ski, M. :
\newblock On the Hausdorff dimension of some fractal sets.
\newblock {\em Studia Math.} 93(2), 155-186 (1989)

\bibitem{hu1993fractal}
Hu, T.Y.,  Lau, K.S.:
\newblock Fractal dimensions and singularities of the Weierstrass type functions.
\newblock {Trans. Amer. Math. Soc} 335(2), 649-665 (1993)

\bibitem{mauldin1986hausdorff}
Mauldin, R. D.,  Williams, S.C.:
\newblock On the Hausdorff dimension of some graphs.
\newblock {Trans. Amer. Math. Soc} 298(2), 793-803 (1986)


\bibitem{mauldin1986hausdorff}
Mauldin, R. D.,  Williams, S.C.:
\newblock On the Hausdorff dimension of some graphs.
\newblock {Trans. Amer. Math. Soc} 298(2), 793-803 (1986)

\bibitem{ledrappier1992dimension}
Ledrappier, F.:
\newblock On the dimension of some graphs.
\newblock {Contemp. Math.} 135, 285-293 (1992)

\bibitem{solomyak1995random}
Solomyak, B.:
\newblock On the random series $\sum\pm$$\lambda^n$ (an Erdos problem).
\newblock {Ann. Math.} 142, 611-625 (1995)

\bibitem{hochman2014self}
Hochman, M.:
\newblock On self-similar sets with overlaps and inverse theorems for entropy.
\newblock {Ann. Math.} 180(2), 773-822 (2014)

\bibitem{varju2019dimension}
Varj{\'u}, P.:
\newblock On the dimension of Bernoulli convolutions for all transcendental parameters.
\newblock {Ann. Math.} 189(3), 1001-1011 (2019)

\bibitem{mandelbrot1977fractals}
Mandelbrot, B.:
\newblock Fractals: Form, Chance and Dimension. WH Freeman and Co.
\newblock {San Francisco} (1997)

\bibitem{hunt1998hausdorff}
Hunt, B.:
\newblock The Hausdorff dimension of graphs of Weierstrass functions.
\newblock {Proc. Am. Math. Soc.}126(3),791-800 (1998)

\bibitem{baranski2014dimension}
Bara{\'n}ski, K., B{\'a}r{\'a}ny, B., Romanowska, J.:
\newblock On the dimension of the graph of the classical Weierstrass function.
\newblock {Adv. Math.}265,32-59 (2014)

\bibitem{tsujii2001fat}
Tsujii, M.:
\newblock Fat solenoidal attractors.
\newblock {Nonlinearity.}14(5),1011-1027 (2001)

\bibitem{shen2018hausdorff}
Shen, W.:
\newblock Hausdorff dimension of the graphs of the classical Weierstrass functions.
\newblock {Math.Z.}289(1),223-266 (2018)

\bibitem{barany2019hausdorff}
B{\'a}r{\'a}ny, B., Hochman, M.,Rapaport, A.:
\newblock Hausdorff dimension of planar self-affine sets and measures.
\newblock {Invent. Math.}216(3),601-659 (2019)

\bibitem{ren2021dichotomy}
Ren, H., Shen, W.:
\newblock A Dichotomy for the Weierstrass-type functions.
\newblock {Invent. Math.}226(3),1057-1100 (2021)



\end{thebibliography}

\end{document}